\documentclass[a4paper,10pt]{article}
\usepackage[title]{appendix}
\usepackage[T1]{fontenc}
\usepackage[utf8]{inputenc}
\usepackage{tablefootnote}
\usepackage{latexsym,amssymb}
\usepackage{biblatex}
\bibliography{ref} 
\usepackage{csquotes}
\usepackage{comment}
\usepackage{graphicx}
\usepackage{hyperref}
\usepackage{amsmath}
\usepackage[english]{babel}
\usepackage{tikz-cd}
\usepackage{scalerel,stackengine}
\usepackage{bm}
\usepackage{url}
\usepackage{amsmath}
\usepackage{amssymb}
\usepackage{amsthm}
\usepackage{mathtools}
\usepackage{amsmath,calligra,mathrsfs}
\usepackage{yhmath}
\usepackage{bashful}
\usepackage{hyperref}
\newtheorem{theorem}{Theorem}
\newtheorem{lemma}[theorem]{Lemma}
\newtheorem{proposition}[theorem]{Proposition}
\newtheorem{corollary}[theorem]{Corollary}

\theoremstyle{definition} 

\newtheorem{example}[theorem]{Example}
\newtheorem{remark}[theorem]{Remark}
\newtheorem{conjecture}[theorem]{Conjecture}

\makeatletter
\newcommand*{\inlineequation}[2][]{%
	\begingroup
	\refstepcounter{equation}%
	\ifx\\#1\\%
	\else
	\label{#1}%
	\fi
	\relpenalty=10000 %
	\binoppenalty=10000 %
	\ensuremath{%
		#2%
	}%
	~\@eqnnum
	\endgroup
}
\makeatother

\title{ \bf{Asymptotic Fermat for signatures \texorpdfstring{$(r,r,p)$}{TEXT}  using the modular approach}
	\noindent}
\newcommand{\rep}{\overline{\rho}}
\newcommand{\fQ}{\mathfrak{Q}}
\newcommand{\fq}{\mathfrak{q}}
\newcommand{\fP}{\mathfrak{P}}
\newcommand{\fp}{\mathfrak{p}}

\newcommand{\cN}{\mathcal{N}}
\newcommand{\cO}{\mathcal{O}}
\newcommand{\N}{\text{Norm}}
\newcommand{\z}{\zeta_r}
\newcommand{\p}{ {\mathfrak{P}}}
\newcommand{\Ker}{\text{Ker}}
\author{Diana Mocanu}
\date{}
\begin{document}
	
	\maketitle
	\begin{abstract}
		Let $K$ be a totally real field, and $r\geq 5$ a fixed rational prime. In this paper, we use the modular method as presented in the work of Freitas and Siksek to study non-trivial, primitive solutions $(x,y,z) \in \mathcal{O}_K^3$ of the signature $(r,r,p)$ equation $x^r+y^r=z^p$ (where $p$ is a prime that varies). 
		An adaptation of the modular method is needed, and we follow the work of Freitas which constructs Frey curves over totally real subfields of $K(\z)$. 
		
	When $K=\mathbb{Q}$ we get that for most of the primes
$5\leq r<150$ with $r\equiv 3,5\mod 8$ there are no
non-trivial, primitive integer solutions $(x,y,z)$ with $2|z$
for signatures $(r,r,p)$ when $p$ is sufficiently large. Similar results hold for quadratic fields, for example when $K=\mathbb{Q}(\sqrt{2})$ there are no non-trivial, primitive solutions $(x,y,z)\in \cO_K^3$ with $\sqrt{2}|z$ for signatures
    $(5,5,p)$, $(11,11,p),(13,13,p)
    $
   and sufficiently large $p$.  
	\end{abstract}    
\section{Introduction}
The \textit{modular method} for solving Diophantine equations was pioneered by Frey, Serre, Ribet, and Wiles with the famous proof of Fermat's Last Theorem in 1995 \cite{AW95}, \cite{Taylor}. Since then, various generalizations of Fermat’s Last Theorem have been
considered, studying equations of the shape
\begin{equation} \label{intro}
	Ax^p+By^q=Cz^r \text{ with } 1/p+1/q+1/r<1
\end{equation}
for fixed integers $A,B$ and $C$. We call $(p,q,r)$ \textit{the signature} of the equation \eqref{intro}. We call a solution $(x,y,z)$ \textit{primitive} if $\gcd(x,y,z)=1$. Note that if $A=B=C=1$, this is equivalent to $x,y,z$ being pairwise coprime. 
A \textit{non-trivial} solution $(x,y,z)$ is a solution where $xyz \neq 0$. The so-called \textit{Generalized Fermat equation} \eqref{intro} is the subject to the following conjecture, which is known to be a consequence of the $ABC$-conjecture.

\begin{conjecture} \label{catalan}
	Fix $A,B,C\in \mathbb{Z}$ pairwise coprime. Then, there exist  only finitely many non-trivial, primitive
	triples  with $(x,y,z)\in \mathbb{Z}^3$ and $p,q,r$ primes such that \eqref{intro} holds (here solutions where
	one of $x, y, z$ equals $1$ are counted only once, e.g. $2^3 +1^q = 3^2$
	for all $q$ counts as one solution.)
\end{conjecture}
A partial result towards this conjecture due to Darmon-Granville \cite{DG} asserts the following.
\begin{theorem}[Darmon-Granville]\label{dar}
	For $A, B, C$ fixed as
	above and a fixed signature $(p, q, r)$ there exist only
	finitely many non-trivial, primitive integer solutions to \eqref{intro}.
\end{theorem}
 The proof involves Faltings' Theorem on the finiteness of rational points on curves of genus at least $2$. Theorem \ref{dar} can be easily generalized for equations with solutions in the ring of integers of any number field $K$, and a suitable version of Conjecture \ref{catalan} is expected to hold. In this setting, we say a solution $(x,y,z)$ is primitive if $x\cO_K+y\cO_K+z\cO_K =\cO_K $.
 
Furthermore, Conjecture \ref{catalan} has been established for many families of signatures (both over $\mathbb{Q}$ and over totally real number fields) using variants of the modular approach.
Some examples of signatures that have been (partially) solved using the modular method can be found in Table \ref{my-table}.

\begin{table}
\begin{tabular}{|l|l|l|}
\hline
Signatures         & Over $\mathbb{Q}$       & Over totally real fields \tablefootnote{These are asymptotic results (i.e. for $n$ large enough). The upper script $^*$ stands for partial results where local conditions on the solutions are assumed.} \\ \hline
$(n,n,n), n\geq 3$ & Wiles \cite{AW95}, Taylor-Wiles \cite{Taylor} & Freitas-Siksek \cite{SN}              \\ \hline
$(n,n,2), n\geq 4$ & Darmon-Merel \cite{DM}, Poonen \cite{P}       & \begin{tabular}[c]{@{}l@{}}Işık, Kara, Özman*\cite{turkish}, \\ Mocanu*\cite{eu}\end{tabular}    \\ \hline
$(n,n,3), n\geq 3$ & Darmon-Merel \cite{DM}, Poonen \cite{P}       & Mocanu* \cite{eu}                          \\ \hline
$(4,2,n), n\geq 4$ & Ellenberg\cite{eisen}, Bennett-Ellenberg-Ng \cite{BEN}   & Torcomian* \cite{tarcomian}                   \\ \hline
\end{tabular}
\caption{Examples of solved signatures using the modular method}
\label{my-table}
\end{table}
The modular method is a strategy of attacking Diophantine Equations (usually parameterised by a varying prime exponent $p$) which can be summarised in three steps.
\begin{itemize}
	\item[\textbf{Step 1.}] \textbf{Constructing a Frey elliptic curve.} Attach to a putative solution (of some Diophantine equation) an elliptic curve $E/K$, for $K$ an appropriately chosen totally real number field. We require $E$ to have the minimal discriminant $\Delta=CD^p$, where $C$ is a constant. 
	\item[\textbf{Step 2.}]\textbf{Modularity/Level lowering.} Prove modularity of $E/K$ and irreducibility of the residual Galois representations $\bar{\rho}_{E,p}$ attached to $E$, to conclude (via level lowering results), that $\bar{\rho}_{E,p} \sim \bar{\rho}_{\mathfrak{f},p}$ where $\mathfrak{f}$ is a Hilbert newform with rational Hecke eigenvalues and level $\mathcal{N}_p$, depending only on $C$. 
 \item[\textbf{Step 3.}]\textbf{Contradiction.} Prove that among the finitely many Hilbert newforms predicted above, none of them corresponds to $\bar{\rho}_{E,p}$.

 \end{itemize}
In Sections \ref{CFEC} and \ref{Section3}, we carefully explain these steps and present the theoretical background needed for this variant of the modular approach to work.  

Let us focus on signature $(r,r,p)$. Usually, the equation considered is
\begin{equation}\label{key}
	x^r+y^r=dz^p
\end{equation}	
where $r,p$ are rational primes, $d$ is a positive integer, with $r$ and $d$ fixed. In this paper, we will follow the recipes described by Freitas in \cite{F} to construct the desired Frey curve. Freitas describes a framework for attacking 
\eqref{key}, by constructing the so-called multi-Frey family of curves. However, we will only work with one elliptic curve belonging to this family, as described in Section \ref{CFEC}.

Using the multi-Frey approach, Freitas shows in \cite{F} that for $p>(1+3^{18})^2$, there are no non-trivial, primitive integer solutions to the $(7,7,p)$ equation when $d=3$. Later on, the same author together with Billerey, Chen, Dieulfait gives an optimal bound for the same result together with a resolution to $(7,7,p)$ with $d=1$ conditional on local constraints \cite{BCDFN2}.
The same authors together with  Dembele prove that for all $p$ there are no non-trivial primitive solutions for signatures $(5,5,p)$ and $(13,13,p)$ when $d=3$ in \cite{BCDF19},\cite{BCDDF22}. Moreover, by assuming local conditions on the putative solution, a partial result is given for $(5,5,p)$ when $d \in \{ 1,2\}$.

Another popular direction of attacking this problem originated in Darmon's program and involves using Frey abelian varieties of higher dimensions (in place of Frey elliptic curves). In a recent paper, Billerey, Chen, Dieulfait, Freitas \cite{BCDFN} use some of the additional structure of the Frey abelian varieties to get asymptotic resolutions for $(11,11,p)$ with $d=1$ for solutions locally away from $xy=0$. In order to accomplish this, the authors had to build on the progress surrounding the modular method by generalizing the theoretical machinery involved in Step 2 to this setting.  

An even more recent result by Freitas and Najman \cite{FNaj} gives the following.
\begin{theorem}
 Let $r, p$ and $d$ as in (\ref{key}) with $r,p\nmid d$.Moreover assume $d$ is an odd positive integer, not an $r$-th power, and $p$ large compared to $r$. Then, for a set of primes $p$ with positive density, the equation \eqref{key} has no non-trivial, primitive integer solutions $(x,y,z)$ with $2|(x+y)$ or $r|(x+y)$. 
\end{theorem}

We say the  \textit{asymptotic Fermat} holds for signature $(r,r,p)$ if there exists a constant $B_r$ such that, there are no non-trivial, primitive integer solutions (respecting certain local conditions) to the equation $x^r+y^r=z^p$ for $p>B_r$.

 Moreover, we examine the equation \eqref{key} over a totally real number field $K$, and give computable criteria of testing if the \textit{asymptotic Fermat} for signature $(r,r,p)$ with non-trivial, primitive solutions in $\mathcal{O}_K$ (respecting certain local conditions)  holds. The argument is pioneered by Freitas and Siksek in \cite{SN} and involves modularity of elliptic curves over totally real fields, an "Eichler-Shimura"-type result, image of inertia comparison, and the study of $S$-unit equations.

	\subsection{Our Results}
	Let $K$ be a number field. For any rational prime $p$ we denote by $$S_{K,p}:=\{\mathfrak{P}\: : \mathfrak{P} \text{ is a prime of } K \text{ above } p\}.$$
	Moreover, for a fixed prime $r$, we denote by $K^+_{r}:=K(\z+\z^{-1})$, where $\z$ is a primitive $r^{\text{th}}$ root of unity. If the prime $r$ is understood from the context, we will omit the subscript $r$ in $K^+_{r}$. 
	
	We would like to study the solutions of $x^r+y^r=z^p$ over a totally real field $K$ via the \textit{modular approach}. The first step is to construct a Frey elliptic curve $E$. One way to realize this is by defining $E/K_r^+$ as described in section \ref{CFEC}. Therefore, it makes sense to view our equation over $K^+_r$ first. Note that if $K$ is totally real, then $K^+_{r}$ stays totally real. Given any number field $F$, we denote by $h_F^+$ the narrow class number of $F$. 
Let $H_F^+$ be the narrow Hilbert class field of $F$ (i.e. the maximal field extension of $F$ which is abelian and unramified at all finite primes).
 By standard class field theoretic results, $h_F^+=[H_F^+:F]$.

	\begin{theorem} \label{main}
		Let $r\geq 5$ be a rational prime and $K$ a totally real number field. Define $K^+:=K(\z+\z^{-1})$  and $S_{K^+}:=S_{K^+,2}\cup S_{K^+,r}$.
		Suppose that there exists a prime $ {\mathfrak{P}} \in S_{K^+,2}$, such that every solution $(\lambda, \mu) \in \mathcal{O}_{S_{K^+}}^* \times \mathcal{O}_{S_{K^+}}^*$ to the $S_{K^+}$-unit equation
		$$ \lambda+\mu = 1
		$$
		satisfies $\max(|v_{ {\mathfrak{P}}}(\lambda)|,|v_{ {\mathfrak{P}}}(\mu)|)\leq 4v_{ {\mathfrak{P}}}(2)$. Then, there is a constant $B_{K,r}$ (depending only on $K$) such that for each rational prime $p>B_{K,r}$, the equation $x^r+y^r=z^p$ has no non-trivial, primitive solutions $(x,y,z) \in \mathcal{O}_{K^+}^3$ with $ {{\mathfrak{P}}}|z$.
		
		In particular, by considering $\fQ$ to be the prime in $S_{K,2}$ below $ \mathfrak{P}$, the equation $x^r+y^r=z^p$ with $p>B_{K,r}$ has no non-trivial, primitive solutions $(x,y,z) \in \mathcal{O}_{K}^3$ with $\fQ|z$.
		
	\end{theorem}
\begin{remark}
	\begin{itemize}
	    \item If we assume modularity of elliptic curves over totally real number fields, the constant $B_{K,r}$ is effectively computable.
	    \item By Siegel \cite{Siegel}, $S$-unit equations have finitely many solutions over number fields. Moreover, they are effectively computable, for example, an $S$-unit solver has been implemented in the free open-source mathematics software, Sage by A. Alvarado, A. Koutsianas, B. Malmskog, C. Rasmussen, D. Roe, C. Vincent, M. West in \cite{AKMRVW}. 
	
	\end{itemize}	
\end{remark}
\begin{theorem}\label{main2}
Let $r\geq 5$ a rational prime, $K$ be a totally real number field and $K^+:=K(\zeta_r+\zeta_r^{-1})$. Denote by $\pi_r:= \z + \z^{-1} -2$.
	Assume that
	\begin{enumerate}
	    \item $r$ is inert in $K$;
	    \item there is a unique prime $\fP$ above $2$ in $K^+$ which has ramification index $e:=e(\fP/2)$;
	    \item $2 \nmid h_{K^+}^+$;
	    \item the congruence $\pi_r 
\equiv \nu^2 \mod \mathfrak{P}^{(4e+1)}$ has no solutions in $\nu \in \cO_{K^+}/\fP^{(4e+1)}$.
	\end{enumerate}
Then, there is a constant $B:=B_{K,r}$ (depending only on $r$ and $K$) such that for each rational prime $p>B$, the equation $x^r+y^r=z^p$ has no non-trivial, primitive solutions $(x,y,z)\in \mathcal{O}_{K^+}^3$ with $\mathfrak{P}|z$.
\end{theorem}	
	\begin{corollary}\label{mainQ}
Fix $r\geq 5$ a rational prime such that $r\equiv 3,5\mod 8$.
Let $\mathbb{Q}^+:=\mathbb{Q}(\zeta_r+\zeta_r^{-1})$, suppose that
$2$ is inert in $\mathbb{Q}^+$ and $2\nmid h_{\mathbb{Q}^+}^+$.
Then, there is a constant $B_r$ (depending only on $r$) such that
for each rational prime $p>B_r$, the equation $x^r+y^r=z^p$ has
no non-trivial, primitive, integer solutions with $2|z$.
\end{corollary}

\begin{corollary}\label{example}
There are no non-trivial, primitive, integer solutions to
\(
x^r+y^r=z^p
\)
with $2|z$ for
\[
r\in\{5,11,13,19,37,53,59,61,67,83,101,107,131,139,149\}
\]
and $p>B_r$, where $B_r$ is a constant depending on $r$.
\end{corollary}

Let's now study $x^r+y^r=z^p$ over quadratic fields of the form
$\mathbb{Q}(\sqrt{d})$ for $d$ square-free.

\begin{corollary}\label{mainQd}
Fix $r\geq 5$ a rational prime and let $K:=\mathbb{Q}(\sqrt{d})$
with $d>1$ a square-free, positive integer and
$K^+:=K(\zeta_r+\zeta_r^{-1})$. Assume that
\begin{enumerate}
    \item $r$ is inert in $K$ and
    $(-1)^{(r-1)/2}r\not\equiv 1,d\mod 8$;
    \item there is a unique prime $\fP$ above $2$ in $K^+$ and
    we denote the unique prime of $K$ below it by $\fQ$;
    \item $2\nmid h_{K^+}^+$.
\end{enumerate}
Then, there is a constant $B_{K,r}$ (depending only on $r$ and $K$)
such that for each rational prime $p>B_{K,r}$, the equation
$x^r+y^r=z^p$ has no non-trivial, primitive solutions
$(x,y,z)\in\cO_K^3$ with $\fQ|z$.
\end{corollary}

\begin{example}\label{exampleQ2}
One can check the conditions in the hypothesis of
Corollary \ref{mainQd} (for example by using Magma) to get the
following. When $K=\mathbb{Q}(\sqrt{2})$ and
$\fQ=(\sqrt{2})\cO_K$, there are no non-trivial, primitive
solutions $(x,y,z)\in\cO_K^3$ with $\fQ|z$ for signatures
\[
(5,5,p),(11,11,p),(13,13,p)
\]
and sufficiently large $p$.
\end{example}

\textbf{Acknowledgements.} I am sincerely grateful to my supervisor Samir Siksek for his continuous support and useful discussions. I also thank my PhD thesis examiners, Tobias Berger and Sam Chow, for their helpful comments and suggestions, and Nuno Freitas for pointing out
an error in the previous versions of Corollaries~\ref{mainQ}
and~\ref{example}. Finally, I thank the anonymous referees
from {Research in Number Theory}.

\subsection{Notational conventions}
	Let $K$ be a totally real field and $E/K$ be an elliptic curve of conductor $\mathcal{N}_E$. Let $p$ be a rational prime. Define
	the following quantities:
	\begin{equation}\label{condNM}
		\mathcal{M}_p=\prod_{\substack{\mathfrak{q}||\mathcal{N}_E \\ p|v_{\mathfrak{q}}(\Delta_\mathfrak{q})}}\mathfrak{q}, \qquad \mathcal{N}_p=\frac{\mathcal{N}_E}{\mathcal{M}_p}
	\end{equation}
	where $\Delta_{\mathfrak{q}}$ is the discriminant of a local minimal model for $E$ at $\mathfrak{q}$.
	
	Let $G_K=\text{Gal}(\bar{K}/K)$, where $\bar{K}$ is the algebraic closure of the number field $K$. For an elliptic curve $E/K$, we write 
	\begin{equation*}
		\overline{\rho}_{E,p}: G_K \to \text{Aut}(E[p])\simeq \text{GL}_2(\mathbb{F}_p)
	\end{equation*}
	for the representation of $G_K$ on the $p$-torsion of $E$. 
	
	For a Hilbert eigenform $\mathfrak{f}$ over $K$, we let $\mathbb{Q}_{\mathfrak{f}}$ denote the field generated by its eigenvalues. A comprehensive definition of \textit{Hilbert modular forms} and their associated representation can be found, for example in Wiles' \cite{AW}.
	
	Let $\mathcal{O}_K$ be the ring of integers of $K$, and $S$ a finite set of prime ideals in $\mathcal{O}_K$. We denote by $$\mathcal{O}_S:=\{ \alpha \in K : v_{\fp}(\alpha) \geq 0 \text{ for all }\fp \notin  S\}$$ the $S$-\textit{integers} of $K$. Moreover, $\cO_S^*$ will be its unit group. An ideal $I$ of $\cO_K$ is called a \textit{prime-to}-$S$-\textit{ideal} if its prime decomposition contains no primes in $S$.
	
 	\section{Constructing Frey Elliptic Curves}\label{CFEC}
	
	\subsection{Diophantine Equations Related to \texorpdfstring{$x^r+y^r=z^p$}{TEXT}}
	Let $r\geq 5$ be a fixed rational prime and $K$ a totally real number field. Consider the equation
	\begin{equation}\label{rrp}
		x^r+y^r=z^p
	\end{equation}
	viewed over $K^+:=K(\zeta_r+\zeta_r^{-1})$ which, as we noted, is totally real. Recall the notation $S_{K,p}:=\{\mathfrak{P}\: : \mathfrak{P} \text{ is a prime of } K \text{ above } p\}.$
	
	In this section, we follow Freitas' ideas in \cite[Section 2]{F} to relate the primitive solutions of \eqref{rrp} to primitive solutions of several homogeneous Diophantine equations defined over $K^+$. We note that in this case $(x,y,z)$ primitive solution implies that $x,y,z$ are pairwise coprime. We write
	\begin{equation}\label{phi}
		\phi_r(x,y) := \frac{x^r+y^r}{x+y} = \sum_{i=0}^{r-1}(-1)^ix^{r-1-i}y^i.
	\end{equation}
	Over the field $L:=K(\z)$ one gets the factorization
	\begin{equation}\label{phi1}
		\phi_r(x,y)= \prod_{i=1}^{r-1}(x+\z^iy).
	\end{equation}
   \begin{proposition} \label{propfact}
		Suppose that $(x,y,z)$ is a primitive, non-trivial solution of \eqref{rrp} in $\cO_{L}^3$. Then, any two factors $x+\z^iy$ and $x+\z^jy$ with $0\leq i<j \leq r-1$ are coprime outside $S_{L,r}$.
	\end{proposition}
	
	\begin{proof}
		Let $\fp \notin S_{L,r}$ be a prime of $L$. Suppose by a contradiction that $\fp | (x+ \z^iy)$ and $\fp | (x+ \z^jy)$ for $0\leq i<j \leq r-1$. Then
		$$ \fp | (x+ \z^iy)-(x+ \z^jy)=\z^i(1-\z^{j-i})y.
		$$
		As $\z^i(1-\z^{j-i})$ is an $S_{L,r}$-unit and $\fp \notin S_{L,r}$, it follows that $\fp|y$. This and the fact that $\fp | (x+ \z^iy)$, implies that $\fp |x$ contradicting the fact that $x$ and $y$ are coprime.
	\end{proof}
	Since $r \geq 5$ is a prime, $r-1\geq 4$ is even and $\phi_r$ factors over $K^+$ into degree two factors of the form
	\begin{equation}\label{fk}
		f_k(x,y):= x^2+(\z^k+\z^{-k})xy+y^2, \quad 1\leq k \leq \frac{r-1}{2}. 
	\end{equation}
	Moreover, we consider $f_0(x,y)=(x+y)^2$.
	\begin{proposition}\label{fkk}
		Suppose that $(x,y,z)$ is a primitive, non-trivial solution of \eqref{rrp} in $\cO_{K^+}^3$. Then the factors $f_i(x,y)$ and $f_j(x,y)$ are coprime outside $S_{K^+,r}$ for $0\leq i<j \leq \frac{r-1}{2}$ and each can be factorized as follows
		\begin{equation}
			(f_k(x,y))\cO_{K^+}= (\mathcal{Z}_k)^p  \prod_{\substack{\mathfrak{p}_r\in S_{K^+,r}}}\mathfrak{p}_r^{e_{k,\mathfrak{p}_r}}  ,\quad 0\leq k \leq \frac{r-1}{2}
		\end{equation} 
		where $\mathcal{Z}_k$ are prime-to-$S_{K^+,r}$-ideals with $\mathcal{Z}_0|(z)^2\cO_{K^+}$ and $\mathcal{Z}_k|(z)\cO_{K^+}$ for $1\leq k\leq (r-1)/2$, and $e_{k,\mathfrak{p}_r}\geq 0$ for all $0\leq k\leq (r-1)/2$.
	\end{proposition}
	\begin{proof}
	    Firstly we note that $$f_k(x,y)=(x+\z^ky)(x+\z^{r-k}y)$$ for $0\leq k \leq \frac{r-1}{2}$, so by Proposition~\ref{propfact} $f_k(x,y)$ are coprime outside $S_{L,r}$. Thus, they are coprime outside $S_{K^+,r}$. Moreover,
\begin{equation}\label{prodf}
	        (x+y)\prod_{k=1}^{(r-1)/2}f_k(x,y)=z^p.
	    \end{equation}
	   Putting these together, we get the desired factorization. Indeed, for a prime $\fq\notin S_{K^+,r}$ dividing $f_k(x,y)$ with $k\geq 1$, equation \eqref{prodf} gives $v_{\fq}(f_k(x,y))=pv_{\fq}(z)$. If $\fq|f_0(x,y)$, then $\fq|(x+y)$ and the same equation gives $v_{\fq}(x+y)=pv_{\fq}(z)$. Since $f_0(x,y)=(x+y)^2$, it follows that $v_{\fq}(f_0(x,y))=2pv_{\fq}(z)$. Thus $\mathcal{Z}_0|(z)^2\cO_{K^+}$, whereas $\mathcal{Z}_k|(z)\cO_{K^+}$ for $k\geq 1$.
	\end{proof}

We will now define a Frey elliptic curve by generalizing to a totally number field the recipes described in \cite{F}.
	\subsection{Frey Elliptic curve}
		
	We consider a non-trivial, primitive solution $(x,y,z)\in \mathcal{O}_{K^+}^3$ to \eqref{rrp} where $ {\mathfrak{P}}|z$ for a fixed $ {\mathfrak{P}}\in S_{K^+,2}$. By combining \eqref{rrp} and \eqref{phi} we get that 
	\[
	(x+y)\phi_r(x,y)=z^p.
	\]
	By Proposition \ref{propfact} it follows that either $ {\mathfrak{P}} | (x+y)$ or $ {\mathfrak{P}} | \phi_r(x,y)$. 
	Therefore by Proposition \ref{fkk} there exists a $k_1$ such that $0\leq k_1 \leq \frac{r-1}{2}$ such that $\fP|f_{k_1}(x,y)$ and $\fP \nmid f_k(x,y)$ for $k \neq k_1$.
	
	Since $\frac{r-1}{2}\geq 2$ we can fix two more subscripts $0\leq k_2, k_3\leq \frac{r-1}{2}$ such that $k_1,k_2,k_3$ are distinct and consider
	\[ \quad f_{k_1}= x^2+(\z^{k_1}+\z^{-k_1})xy+y^2,\quad f_{k_2}= x^2+(\z^{k_2}+\z^{-k_2})xy+y^2, \quad f_{k_3}=x^2+(\z^{k_3}+\z^{-k_3})xy+y^2
 ,	\]
	We want to find $(\alpha, \beta, \gamma)$ such that 
	\[ \alpha f_{k_1}+\beta f_{k_2}+\gamma f_{k_3} =0.
	\]
	In particular, we can take 
	\begin{equation}\label{abc2}
		\begin{cases}
			
			\alpha&=\zeta_r^{k_3}+\zeta_r^{-k_3}-\zeta_r^{k_2}-\zeta_r^{-k_2},\\
\beta&=\zeta_r^{k_1}+\zeta_r^{-k_1}-\zeta_r^{k_3}-\zeta_r^{-k_3},\\
\gamma&=\zeta_r^{k_2}+\zeta_r^{-k_2}-\zeta_r^{k_1}-\zeta_r^{-k_1}.

		\end{cases}
	\end{equation}
	Write $A_{(x,y)}:=\alpha f_{k_1}(x,y), B_{(x,y)}:=\beta f_{k_2}(x,y), C_{(x,y)}:=\gamma f_{k_3}(x,y) $ and define
	\begin{equation}\label{Frey2}
		E_{(x,y)}: Y^2=X(X-A_{(x,y)})(X+B_{(x,y)}).
	\end{equation}
	Denote $E_{(x,y)}$ by $E$. Note that $E$ is defined over $K^+$. Recall that $S_{K^+}:=S_{K^+,2}\cup S_{K^+,r}$.
	\begin{lemma}
The numbers $A,B, C$ factorize in $\cO_{K^+}$ is as follows
\begin{equation}\label{modlemma}
\begin{aligned}
(A)\cO_{K^+}
&= (\mathcal{Z}_A)^p
\prod_{\substack{\mathfrak{p}_r\in S_{K^+,r}}}
\mathfrak{p}_r^{a_{\mathfrak{p}_r}}, \quad
(B)\cO_{K^+}
= (\mathcal{Z}_B)^p
\prod_{\substack{\mathfrak{p}_r\in S_{K^+,r}}}
\mathfrak{p}_r^{b_{\mathfrak{p}_r}},\\
(C)\cO_{K^+}
&= (\mathcal{Z}_C)^p
\prod_{\substack{\mathfrak{p}_r\in S_{K^+,r}}}
\mathfrak{p}_r^{c_{\mathfrak{p}_r}}.
\end{aligned}
\end{equation}
	where $\mathcal{Z}_A,\mathcal{Z}_B,\mathcal{Z}_C$ are pairwise coprime prime-to-$S_{K^+,r}$-ideals dividing $(z)^2\cO_{K^+}$, $\fP|\mathcal{Z}_A$ with exponents $a_{\mathfrak{p}_r},b_{\mathfrak{p}_r},c_{\mathfrak{p}_r}\geq 0$ as given in Proposition \ref{fkk}.
	\end{lemma}
	\begin{proof}
	First note that $\alpha, \beta, \gamma$ can be written in the form $\pm \z (1-\z^t)(1-\z^u)$, where
neither $t$ nor $u$ are multiples of $r$, which means that the only primes dividing $\alpha\beta\gamma$ are the ones in $S_{K^+,r}$. The result then follows from the definition of $A,B, C$ and Proposition \ref{fkk}.
	\end{proof}
	\begin{remark}
	
	We note that the Frey elliptic curve $E$ depends on the prime $p$ as the coefficients $A,B$ and $C$ do.
	    
	\end{remark}\label{note}
	\subsection{Arithmetic Invariants}
	It is a standard result that an elliptic curve $E/K$ defined as
	$$ E: Y^2=X(X-A)(X+B)
	$$
	 with $A+B+C=0$ has the corresponding arithmetic invariants
	\begin{equation}\label{invariants}
		\Delta_E=2^4(ABC)^2, \qquad 
		c_4= -2^4(AB+BC+AC), \qquad
		j_E= -2^8
		\frac{(AB+BC+AC)^3}{(ABC)^2}.
	\end{equation}
	Recall $S_{K^+}:=S_{K^+,r}\cup S_{K^+,2}$ where for a rational prime $p$, we have
 $$S_{K^+,p}:=\{\mathfrak{P}\: : \mathfrak{P} \text{ is a prime of } K^+ \text{ above } p\}.$$
	
	\begin{proposition}\label{lemmane}
		Let $(x,y,z)$ be a primitive non-trivial solution to \eqref{rrp}. Then, the conductor of $E$ is 
		\[ \mathcal{N}_E = \displaystyle\prod_{\mathfrak{P}\in S_{K^+,2}}\mathfrak{P}^{e_{\mathfrak{P}}} \prod_{\substack{\mathfrak{q}\in S_{K^+,r}}}\mathfrak{q}^{f_{\mathfrak{q}}} \prod_{\substack{ \fp | ABC  \\ \mathfrak{p}\notin S_{K^+}}}\mathfrak{p}
		\]
		Then, for all primes $\fp\notin S_{K^+}$, the model $E$ is minimal, semistable and satisfies $p|v_{\mathfrak{p}}(\Delta_E)$. Moreover
		\[ \mathcal{N}_p = \displaystyle\prod_{\mathfrak{P}\in S_{K^+,2}}\mathfrak{P}^{e_{\mathfrak{P}}'} \prod_{\substack{\mathfrak{q}\in S_{K^+,r}}}\mathfrak{q}^{f_{\mathfrak{q}}'} 
		\]
		where $0\leq e_{\mathfrak{P}}'\leq e_{\mathfrak{P}}\leq 2+6v_{\fP}(2)$ and $0\leq f_{\mathfrak{q}}'\leq f_{\mathfrak{q}} \leq 2$ and $ \mathcal{N}_E,  \mathcal{N}_p$ are as in (\ref{condNM}).
	\end{proposition}
\begin{proof}
   For each $\fp \in S_{K^+}$, the bounds on the exponents follow from \cite[Theorem IV.10.4]{silv}.
   Hence, it is enough to show that for $\fp \notin S_{K^+}$ the following holds
   \begin{enumerate}
       \item if $\fp \nmid ABC$, the model $E$ is minimal and has good reduction at $\fp$;
       \item if $\fp | ABC$ the model $E$ is minimal and has multiplicative reduction at $\fp$ and moreover $p|v_{\fp}(\Delta_E)$.
   \end{enumerate}
   Then, the result will follow from the definitions of $\mathcal{N}_E$ and $\mathcal{N}_p$ given in (\ref{condNM}).
   
   In order to prove (i) and (ii) we note that a prime $\fp \notin S_{K^+}$ we have that
   \begin{equation} \label{tra}
       v_{\fp}(\Delta_E)=2v_{\fp}(ABC)=2pv_{\fp}(\mathcal{Z}_A\mathcal{Z}_B\mathcal{Z}_C).
   \end{equation}
   If  $\fp \nmid ABC$, then $v_{\fp}(\Delta_E)=0$. Thus, the model is minimal and $E$ has good reduction at $\fp$. If $\fp | ABC$, then $\fp$ divides precisely one of $A$, $B$ and $C$ as they are pairwise coprime outside $S_{K^+,r} \subset S_{K^+}$. Hence, $v_{\fp}(c_4)=0$, giving a minimal model for $E$ with multiplicative reduction at $\fp$. By (\ref{tra}) we see that $p|v_{\fp}(\Delta_E)$.

\end{proof}

	\section{Theoretical Background}\label{Section3}
	\subsection{Modularity}
	Let $K$ be a totally real field and $E$ an elliptic curve over $K$, we say that $E$ is \textit{modular} if there exists a Hilbert cuspidal eigenform $\mathfrak{f}$ over $K$ of parallel weight $2$, with rational Hecke
	eigenvalues, such that the Hasse–Weil L-function of $E$ is equal to the Hecke L-function of $\mathfrak{f}$. In particular, this implies that the mod $p$ Galois representations are isomorphic, which we denote by $\overline{\rho}_{E,p} \sim \overline{\rho}_{\mathfrak{f},p}$.

	We will use the following modularity theorem proved by Freitas, Hung, and Siksek in \cite{FHS}:
\begin{theorem} \label{modularity}
Let $K$ be a totally real field. There are at
most finitely many $\bar{K}$- isomorphism classes of non-modular elliptic curves $E$ over $K$. Moreover, if $K$ is real
quadratic, then all elliptic curves over $K$ are modular.
\end{theorem}
   Furthermore, Derickx, Najman, and Siksek have recently proved in \cite{DNS}:
\begin{theorem} \label{modularity2}
Let $K$ be a totally real cubic number field and $E$ be an elliptic curve over $K$. Then $E$ is
modular.
\end{theorem}
\begin{corollary}
    \label{modularity3}
Let $E_p$ be the Frey curve defined by \eqref{Frey2} (which has a dependency on $p$ as noted in Remark \ref{note}). By construction, $E_p$ is defined over the totally real field $K^+$. Then, there is some constant $A_{K^+}$ depending only on $K^+$, such that $E_p$ is modular whenever $p>A_{K^+}$.
\end{corollary}
\begin{proof}
By Theorem \ref{modularity}, there are at most finitely many possible $\overline{K^+}$-isomorphism classes of elliptic curves over $K^+$ which are not modular. Let $E$ be the elliptic curve defined in (\ref{Frey2}). Let $j_1, j_2, \dots, j_n \in K^+$ be the $j$-invariants of these classes. Define $\lambda := -B/A$. The $j$-invariant of $E$ is
$$j(\lambda)=2^8(\lambda^2-\lambda+1)^3\lambda^{-2}(\lambda-1)^{-2}
.$$

Each equation $j(\lambda)=j_i$ has at most six solutions $\lambda \in K^+\setminus\{0,1\}$. Thus there are values $\lambda_1,\dots,\lambda_m \in K^+\setminus\{0,1\}\: (\text{where }m\leq 6n)$ such that if $\lambda \neq \lambda_k$ for all $k$, then the elliptic curve $E$ with $j$-invariant $j(\lambda)$ is modular.
If $\lambda = \lambda_k$ then $-B/A = \lambda_k$. Hence, as fractional ideals,
\begin{equation}\label{1}
(-B/A)\cO_{K^+} = (\lambda_k)\cO_{K^+}.
\end{equation}
By (\ref{modlemma}) we get that
\begin{equation}\label{2}
(-B/A)\cO_{K^+} = (\mathcal{Z})^p
\prod_{\mathfrak{p}_r\in S_{K^+,r}}
\mathfrak{p}_r^{l_{\mathfrak{p}_r}},
\end{equation}
where $\mathcal{Z}:=\mathcal{Z}_B\mathcal{Z}_A^{-1}$ is a prime-to-$S_{K^+,r}$ fractional ideal, and $l_{\mathfrak{p}_r}:=b_{\mathfrak{p}_r}-a_{\mathfrak{p}_r}$ are integers. Since $\fP\mid\mathcal{Z}_A$ and $\fP\nmid\mathcal{Z}_B$, equations \eqref{1} and \eqref{2} give $v_{\fP}(\lambda_k)=pv_{\fP}(\mathcal{Z})=-pv_{\fP}(\mathcal{Z}_A)<0$. Thus $p\mid v_{\fP}(\lambda_k)$ and hence $p\leq |v_{\fP}(\lambda_k)|$. As $\lambda_k$ is fixed, this gives an upper bound on $p$ for each $k$. Taking the maximum of these absolute values over all $k$ and all primes $\fP$ above $2$, together with $0$, gives a constant $A_{K^+}$ depending only on $K^+$. Therefore, if $p>A_{K^+}$, the curve $E$ is modular.
\end{proof}

\begin{remark}
The constant $A_{K^+}$ is ineffective as the finiteness of Theorem \ref{modularity} relies on Falting's Theorem (which is ineffective). See \cite{FHS} for more details.
Note that if $K^+$ is quadratic or cubic we get $A_{K^+}=0$ (by the last part of Theorem \ref{modularity} and Theorem \ref{modularity2}).
\end{remark}
\subsection{Irreducibility of\texorpdfstring{$\mod{p}$}{TEXT} representations of elliptic curves }\label{irreducibility}
We need the following theorem in the level lowering step of our proof. This was proved in \cite[Theorem 2]{SN1} and it is derived from the work of David \cite{david} and Momose \cite{momose} who in turn built on Merel's Uniform Boundedness Theorem \cite{merel}.
\begin{theorem} \label{irred}
	Let $K$ be a Galois totally real field. There is an effective constant $C_K$, depending only on $K$, such that the following holds. If $p>C_K$ is prime, and $E$
	is an elliptic curve over $K$ which has semistable reduction at all $\mathfrak{q}|p$, then $\overline{\rho}_{E,p}$ is irreducible.
\end{theorem} 
\begin{remark}
	The above theorem is also true for any totally real field by replacing $K$ by its Galois closure.
\end{remark}
	\subsection{Level Lowering}
	We present a level lowering result proved by Freitas and Siksek in \cite{SN} derived from the work of Fujiwara \cite{K}, Jarvis \cite{J}, and Rajaei \cite{R}.
	\begin{theorem} \label{ll}
		
		Let $K$ be a totally real field and $E/K$ be an elliptic curve of conductor $\mathcal{N}_E$, suppose the following statements hold:
		\begin{enumerate}
			\item $p \geq 5$, the ramification index $e(\mathfrak{q}/p)<p-1$ for all $\mathfrak{q}|p$, and $\mathbb{Q}(\zeta_p)^+\nsubseteq K$,
			\item E is modular,
			\item $\overline{\rho}_{E,p}$ is irreducible,
			\item $E$ is semistable at all $\mathfrak{q}|p$,
			\item $p|v_{\mathfrak{q}}(\Delta_{\mathfrak{q}}) $ for all $\mathfrak{q}|p$.
		\end{enumerate}
		Then, there is a Hilbert eigenform $\mathfrak{f}$ of parallel weight $2$ that is new at level $\mathcal{N}_p$ (as defined in \ref{condNM}) and
		some prime $\varpi$ of $\mathbb{Q}_{\mathfrak{f}}$ such that $\varpi|p$ and $\overline{\rho}_{E,p} \sim \overline{\rho}_{\mathfrak{f},\varpi}$.
	\end{theorem}
	\begin{proof}
		See \cite[Corollary 2.2]{SN}.
	\end{proof}
	\subsection{Image of Inertia}
	We gather information about the images of inertia $\rep_{E,p}(I_{\mathfrak{q}})$. This is a crucial step in controlling the behaviour at the primes in $S_{K^+}$ of the newform obtained by level lowering.
	\begin{lemma} \label{lemma}
		Let $E$ be an elliptic curve over $K$ with $j$-invariant $j_E$. Let $p\geq 5$ and
		let $\mathfrak{q} \nmid p$ be a prime of $K$. Then $p | \# \rep_{E,p}(I_{\mathfrak{q}})$ if and only if $E$ has potentially
		multiplicative reduction at $\mathfrak{q}$ (i.e. $v_{\mathfrak{q}}(j_E)<0$) and $p \nmid v_{\mathfrak{q}}(j_E)$.
	\end{lemma}
	\begin{proof}
		See \cite[Lemma 3.4]{SN}.
	\end{proof}

For the rest of this section let $(x,y,z)\in \cO_{K^+}^3$ be a non-trivial, primitive solution to \eqref{rrp} and ${\mathfrak{P}} \in S_{K^+,2}$ with $\fP|z$. We define the Frey curve 
 $E$ as in (\ref{Frey2}).

\begin{lemma}\label{invariantslemma}
Let $E$, $(x,y,z)$ and ${\mathfrak{P}}$ as above. Then
\[
\qquad v_{\fP}(j_E)=
\begin{cases}
8v_{\fP}(2)-4pv_{\fP}(z),&\text{if }k_1=0,\\
8v_{\fP}(2)-2pv_{\fP}(z),&\text{if }k_1\neq0.
\end{cases}
\]
\end{lemma}
\begin{proof}
Firstly, note that $\alpha, \beta, \gamma \in \mathcal{O}_{S_{K^+,r}}^*$, i.e. their factorisation contains only primes above $r$. Moreover, $A,B,C$ are pairwise coprime outside $S_{K^+,r}$ (by Proposition \ref{propfact} and Proposition \ref{fkk}). By the definition of $A$, we see that $\fP|A$, hence $\fP\nmid B,C$. Thus, by (\ref{invariants}) $v_{\fP}(j_E)=8v_{\fP}(2)-2v_{\fP}(A)$. We note that $v_{\fP}(A)=v_{\fP}(f_{k_1}(x,y))$. If $k_1\neq0$, equation \eqref{prodf} gives $v_{\fP}(f_{k_1}(x,y))=pv_{\fP}(z)$. If $k_1=0$, the same equation gives $v_{\fP}(x+y)=pv_{\fP}(z)$, and since $f_0(x,y)=(x+y)^2$, we get $v_{\fP}(f_0(x,y))=2pv_{\fP}(z)$. Substituting these values completes the proof.
\end{proof}
	\begin{lemma} \label{imgi}
		Let $E$, $(x,y,z)$ and ${\mathfrak{P}} $ as above and the prime exponent $p>4v_{\mathfrak{P}}(2)$.Then $E$ has potentially multiplicative reduction at ${\mathfrak{P}}$ and $p | \# \rep_{E,p}(I_{\mathfrak{P}})$.
	\end{lemma}
	\begin{proof}
Assume that $\mathfrak{P}\in S_{K^+,2}$ with $v_{\mathfrak{P}}(z)=k\geq1$. By Lemma \ref{invariantslemma}, $v_{\mathfrak{P}}(j_E)$ is either $8v_{\mathfrak{P}}(2)-4pk$ or $8v_{\mathfrak{P}}(2)-2pk$. As $p>4v_{\mathfrak{P}}(2)$, it follows that $v_{\mathfrak{P}}(j_E)<0$ in both cases. Moreover, $v_{\mathfrak{P}}(j_E)\equiv8v_{\mathfrak{P}}(2)\mod p$. Since $p\geq5$ and $0<v_{\mathfrak{P}}(2)<p$, we get $p\nmid v_{\mathfrak{P}}(j_E)$. This implies that $E$ has potentially multiplicative reduction at $\mathfrak{P}$ and by Lemma \ref{lemma} we get $p|\#\rep_{E,p}(I_{\mathfrak{P}})$.
\end{proof}
	
	\subsection{Eichler-Shimura}\label{ESconj}
	For totally real fields, modularity reads as follows.
	\begin{conjecture}[Eichler-Shimura]
		Let $K$ be a totally real field. Let $\mathfrak{f}$ be a Hilbert
		newform of level $\mathcal{N}$ and parallel weight $2$, with rational eigenvalues. Then there is
		an elliptic curve $E_{\mathfrak{f}}/K$ with conductor $\mathcal{N}$ having the same L-function as $\mathfrak{f}$.
	\end{conjecture}
	Freitas and Siksek \cite{SN} obtained the following theorem from the works of Blasius \cite{DB}, Darmon \cite{D} and Zhang \cite{Z}.
	
	\begin{theorem}\label{inertia}
		Let $E$ be an elliptic curve over a totally real field $K$, and $p$ be an odd prime. Suppose that $\overline{\rho}_{E,p}$
		is irreducible, and $\overline{\rho}_{E,p} \sim \overline{\rho}_{\mathfrak{f},p}$ for some Hilbert newform $\mathfrak{f}$ over $K$ of level $\mathcal{N}$ and parallel weight $2$ which
		satisfies $\mathbb{Q}_{\mathfrak{f}} = \mathbb{Q}$. Let $\mathfrak{q}\nmid p$ be a prime ideal of $\mathcal{O}_K$ such that:
		\begin{enumerate}
			\item $E$ has potentially multiplicative reduction at $\mathfrak{q}$,
			\item $p| \# \rep_{E,p}(I_{\mathfrak{q}})$,
			\item $p \nmid (\text{Norm}_{K/\mathbb{Q}}(\mathfrak{q})\pm 1) $.
		\end{enumerate}
		Then there is an elliptic curve $E_{\mathfrak{f}}/K$ of conductor $\mathcal{N}$ with the same L-function as $\mathfrak{f}$.
	\end{theorem}
\begin{proof}
	See \cite[Corollary 2.2]{SN}.
\end{proof}

	\section{Proof of Theorem \ref{main}}
	Firstly, we put together the first steps of the modular approach to get the following. Recall $S_{K^+}:=S_{K^+,r}\cup S_{K^+,2}$ where for a rational prime $p$, we have
 $$S_{K^+,p}:=\{\mathfrak{P}\: : \mathfrak{P} \text{ is a prime of } K^+\text{ above } p\}.$$
	
	\begin{theorem}[Level Lowering and Eichler Shimura]\label{LLES}
		
		Fix $r \geq 5$ a rational prime. Let $K$ be a totally real number field and define $K^+:=K(\zeta_r+\zeta_r^{-1})$. Suppose there exists a prime $ {\mathfrak{P}} \in S_{K^+,2}$. Then there is a constant $B_{K,r}$
		depending only on $K$ such that the following hold. Suppose $(x,y,z) \in \cO_{K^+}^3$ is a non-trivial, primitive solution to $x^r+y^r=z^p$ with prime exponent $p>B_{K,r}$ such that $ {\fP}|z$. Write $E$
		for the Frey curve (\ref{Frey2}). Then, there exists an elliptic curve $E'$ over $K^+$ such that
		\begin{enumerate}
			\item the elliptic curve $E'$ has good reduction outside $S_{K^+}$;
			\item $\overline{\rho}_{E,p}\sim\overline{\rho}_{E',p}$;
			\item $\# E'(K^+)[2]=4$;
			\item $E'$ has potentially multiplicative reduction at $ \mathfrak{P}$ $(v_{ {\mathfrak{P}}}(j_{E'})<0)$.
		\end{enumerate}
	\end{theorem}
	The proof follows precisely as the one of Theorem 9 in \cite{SN}. It is included for completion in Appendix \ref{appendix}.
	Now we are ready to prove the Theorem \ref{main}.
	\begin{proof}[Proof of Theorem \ref{main}]
	The proof follows precisely as the one of Theorem 3 in \cite{SN}. 
	The idea consists in associating to each non-trivial, primitive solution $(x,y,z)\in \cO_{K^+}^3$ with $\fP|z$ a Frey elliptic curve as described in Theorem \ref{LLES}, which we denote by $E$.
	Let $p>B_{K,r}$ as given in Theorem \ref{LLES}, then $E$ gives rise to an elliptic curve $E'/K^+$ with full two torsion, $j_{E'} \in \mathcal{O}_{S_{K^+}}$ and $v_{\fP}(j_{E'}) < 0$.
	
	The condition on the $S_{K^+}$-units in our hypothesis gives $v_{\fP}(j_{E'}) \geq 0$, a contradiction. This is proven by Freitas and Siksek in \cite[Theorem 3]{SN} by parametrizing all of the curves $E'$ with the above properties. In particular, the last line follows by the fact that any non-trivial, primitive solution $(x,y,z)\in \cO_{K}^3$ with $\fQ|z$ gives rise to a solution $(x,y,z)\in \cO_{K^+}^3$ with $\fP|z$.
	
	\end{proof}
\begin{remark}
We note that even though this theorem resembles Theorem 3 in \cite{SN} we need extra assumptions to make the proof work. 
\begin{itemize}
    \item As the equation $x^r+y^r=z^p$ is not homogeneous (as opposed to $x^p+y^p=z^p$), we need the extra condition $\fP|z$ in order to get the desired potentially multiplicative reduction of $E$ (and hence of $E'$) at the prime $\fP$.
    \item The set $S_{K^+}$ contains primes above $2$ as long as primes above $r$ (in \cite[Theorem 3]{SN} the corresponding set had only primes above $2$). This is because, the Frey curve $E$ resulting from signature $(r,r,p)$ might have bad reduction at primes above $r$, which is inherited by $E'$ (after level lowering and Eichler-Shimura).
\end{itemize}
     
\end{remark}
	\section{Proof of Theorem \ref{main2}}
	Suppose that $K$ is a totally real number field and $r\geq 5$ is a fixed rational prime. By (i) $r$ is inert in $K$. Denote $K^+:=K(\z+\z^{-1})$ and let 
	\begin{equation}\label{pir}
	    \pi_r:=\z+\z^{-1}-2=-(1-\z)(1-\z^{-1}).
	\end{equation}
	Note that $\pi_r$ is the uniformizer of the unique prime above $r$ in $K^+$. By (ii) there is a unique prime $\fP \in S_{K^+,2}$ of ramification index $e:=e(\fP/2)$ and by (iii) $2 \nmid h_{K^+}^+$. Moreover, we have that the congruence $\pi_r 
\equiv \nu^2 \mod \mathfrak{P}^{(4e+1)}$ has no solutions in $\nu \in \cO_{K^+}/\fP^{(4e+1)}$ by (iv). 

We will show that these assumptions guarantee that any solution to an $S_{K^+}$-unit equation satisfies the hypothesis in Theorem \ref{main}, giving the desired conclusion.

	\begin{lemma} \label{square}
		Fix $r \geq 5 $ and $K$ a totally real number field satisfying (i)-(iv) as described above. 
 Consider the natural homomorphism
		\[ \varphi \colon \mathcal{O}_{S_{K^+,r}}^* \to (\mathcal{O}_{K^+}/\mathfrak{P}^{4e+1})^*.\]
		Then, $\Ker({\varphi}) \subseteq \mathcal({O}_{S_{K^+,r}}^*)^2$.
	\end{lemma}
\begin{proof}
	Let $\mu \in\Ker({\varphi})$, so $\mu \equiv 1 \mod \mathfrak{P}^{4e+1}$. We want to show that $\mu$ is a square. Consider  $\pi_r$ as described in \eqref{pir}, a uniformizer for the unique prime above $r$ in $K^+$. Thus, $\mu \in\mathcal{O}_{S_{K^+,r}}^*$, implies that $\mu = \epsilon \pi_r^t$, where $\epsilon \in \mathcal{O}_{K^+}^*$. We first show that $\epsilon$ must be a square. Note that 
	\[ (\zeta_r^{\frac{r-1}{2}} + \zeta_r^{-\frac{r-1}{2}})^2 = \pi_r + 4
	.\]
	Thus, $\pi_r \equiv \alpha^2 \mod 4$ where we denoted $\alpha:=(\zeta_r^{\frac{r-1}{2}} + \zeta_r^{-\frac{r-1}{2}})$. 
	As $\fP$ is the unique prime above $2$, we get that 
	$ (4)\cO_{K^+}=\fP^{2e}$ and hence
	\[
	\pi_r \equiv \alpha^2 \mod \fP^{2e}
	.\]
	This and $\mu \equiv 1 \mod \mathfrak{P}^{4e+1}$ imply that $\epsilon \equiv \beta^2 \mod \fP^{2e}$, where $\beta:=\alpha^{-t}$. As $\alpha$ is a unit, and $t$ an integer, it follows that $\beta \in \cO_{K^+}$. Moreover, the above-mentioned congruence gives $(\beta^2-\epsilon)/4\in \cO_{K^+}$. Suppose that $\epsilon$ is not a square. Consider the field extension $L=K^+(\sqrt{\epsilon})$. We will show that $L$ is unramified at all finite places hence contradicting $2 \nmid h_{K^+}^+$.
	Consider the element $\delta:=\frac{\beta+\sqrt{\epsilon}}{2}$. Its minimal polynomial is
	$$ m_{\delta}(X) = X^2-\beta X+ \frac{\beta^2-\epsilon}{4}.
	$$
	
 This belongs to $\mathcal{O}_{K^+}[X]$ and has discriminant $\Delta= \epsilon \in \mathcal{O}_{K^+}^*$, proving that $L$ is unramified at all the finite places, contradicting $2 \nmid h_{K^+}^+$. Thus, we must have $\epsilon:= \gamma^2$ for some $\gamma \in \mathcal{O}_{K^+}^*$. Putting everything together
	$$ \mu = \epsilon \pi_r^t = \gamma^2 \pi_r^t.
	 $$
	In order to show that $\mu$ is a square, it is enough to show that $t$ is even. So, let's suppose that $t$ is odd, i.e. $t = 2k+1$ for some integer $k$. Hence $\mu = (\gamma \pi_r^{k})^2 \pi_r$. As we assumed 
	$\mu \equiv 1 \mod \fP^{4e+1}$, it must be that case that 
	\[ \pi_r \equiv \nu^2 \mod \fP^{4e+1}
	\]
	for some $\nu \in \mathcal{O}_{K^+}$, contradicting assumption (iv). In conclusion $t$ must be even and so $\mu \in \mathcal({O}_{S_{K^+,r}}^*)^2$
	
\end{proof}
	\begin{proof}[Proof of Theorem \ref{main2}]
		This is an application of Theorem \ref{main}.
	We first note that by using the notation at the beginning of this section we can consider $\fP_r:=(\pi_r)\cO_{K^+}$ to be the unique prime above $r$ in $K^+$ and $\fP$ to be the unique prime above $2$ in $K^+$ where we denote by $e:=e(\fP/2)$ the ramification index of $\fP$. Hence $S_{K^+,r}=\{\fP_r\}$ and $S_{K^+,2}=\{\fP\}$ giving $S_{K^+}=\{\fP,\fP_r\}$.
		
		We prove that the hypothesis of Theorem \ref{main} holds. More precisely, we want to show that every solution $(\lambda, \mu) \in \mathcal{O}_{S_K^+}^* \times \mathcal{O}_{S_K^+}^*$ to the equation
		\begin{equation}\label{sunit}
			\lambda+\mu = 1
		\end{equation} 
		satisfies $\max(|v_{\p}(\lambda)|,|v_{\p}(\mu)|)\leq 4e$.
		
		Suppose by a contradiction we have an $S_{K^+}$-unit  solution $(\lambda, \mu)$ with $|v_{\p}(\lambda)|>4e$. Without loss of generality, we may assume $v_{\p}(\lambda)>4e$. Otherwise, one can consider $(\frac{1}{\lambda}, -\frac{\mu}{\lambda})$ instead, which is also a solution to the equation. By the properties of non-archimedean valuations applied to \eqref{sunit} it follows that $v_{\p}(\mu)=0$. Thus, we deduce that $\mu \equiv 1 \text{ mod } \fP^{4e+1}$. Hence $\mu$ lies in the kernel of the natural homomorphism
		\[ \varphi \colon \mathcal{O}_{S_{K^+,r}}^* \to (\mathcal{O}_{K^+}/\fP^{4e+1})^*.\]
		By Lemma \ref{square}, $\Ker{\varphi} \subseteq \mathcal({O}_{S_{K^+,r}}^*)^2$. 
		
		Thus, for each solution of the $S_{K^+}$-unit equation \eqref{sunit} with
		\begin{equation} \label{val}
			v_{\p}(\lambda)>4e, \quad v_{\p}(\mu)=0
		\end{equation}
	 we get that $\mu = \tau^2$ with $\tau \in \mathcal{O}^*_{S_{K^+},r}$. As there are only finitely many solutions to the $S$-unit equation, we may suppose that $(\lambda, \mu)$ satisfies \eqref{val} with the value of $v_{\p}(\lambda)$ as large as possible.
		
		We can rewrite \eqref{sunit} as 
		\begin{equation}
			\lambda = (1-\tau)(1+\tau)
		\end{equation}
		Denote by $\lambda_1:= 1-\tau$ and $\lambda_2:=1+\tau$ and by $t_i:=v_{\p}(\lambda_i)$ for $i=1,2$. By assumption $t:=v_{\p}(\lambda)>4e$, giving $t=t_1+t_2>4e$. By noting that $$\lambda_1+\lambda_2 = 2$$
		we can only have that $t_1=e$ or $t_2=e$. By changing the sign of $\tau$ if necessary, we may assume that $t_2=e$ and hence $t_1=t-e$. Now, note that by rearranging we get the following $S_{K^+}$-unit relation:
		$$ \frac{\lambda_2^2}{4\tau}+\frac{-\lambda_1^2}{4\tau}=1
		$$
		We compute $v_{\p}(\frac{\lambda_2^2}{4\tau})=2t_2-2e=0$ and $v_{\p}(\frac{-\lambda_1^2}{4\tau})=2t_1-2e=2t-4e=t+(t-4e)>t$ as $t>4e$ by assumption. So we found a new solution $(\lambda',\mu')=(\frac{-\lambda_1^2}{4\tau},\frac{\lambda_2^2}{4\tau})$ to \eqref{sunit} with valuations $v_{\p}(\lambda')> t,v_{\p}(\mu')=0$. This contradicts the maximality of $v_{\p}(\lambda)$ and completes the proof.
	
	\end{proof} 
\section{Proof of Corollaries \ref{mainQ}, \ref{example} and \ref{mainQd}}
We will now show how to apply Theorem \ref{main2} in the cases where $K=\mathbb{Q}$ or $K$ is a real quadratic field. Firstly, we need the following lemma. 
\begin{lemma}\label{sqnorm}
Let $\fP$ be the only prime above $2$ in $K^+$ and $\lambda \in \cO_{K^+}$.
Suppose that $\lambda \equiv \nu^2 \mod \fP^n$ where $\nu \in \cO_{K^+}$. Then
$$ \text{Norm}_{K^+/K}(\lambda) \equiv v^2 \mod \mathfrak{Q}^{\lceil n/e' \rceil}
$$
where $\mathfrak{Q}$ is the unique prime above $2$ in $K$ and $e'$ is the ramification index $e':=e(\fP/\mathfrak{Q})$ and $v\in \cO_{K}.$
\end{lemma}
\begin{proof}
By noting that the congruence $\lambda \equiv \nu^2 \mod \fP^n$ is preserved by taking Galois conjugates (which fix the prime $\fP$), we can take norms to get 
$$\N_{K^+/K}(\lambda) \equiv \N_{K^+/K}(\nu)^2\mod \fP^n.$$
As $\lambda, \nu \in \cO_{K^+}$, it follows that $\N_{K^+/K}(\lambda),\N_{K^+/K}(\nu) \in \cO_K$ and so
$\N_{K^+/K}(\lambda) \equiv v^2\mod \fQ^{\lceil n/e' \rceil}$
where $v:=\N_{K^+/K}(\nu) \in \cO_{K}.$
\end{proof}

Now, we fix a prime $r\geq 5$ and $K$ a totally real number field.
We let $K^+:=K(\z+\z^{-1})$. Then, as in the statement of
Theorem \ref{main2} we assume $r$ is inert in $K$, and $\fP$ is
the unique prime above $2$ in $K^+$ with ramification index
$e:=e(\fP/2)$ and we denote by $\pi_r:=\z+\z^{-1}-2$ and note that $\N_{K^+/K}(\pi_r)=(-1)^{(r-1)/2}r.$
For the rest of this section, denote by $q:=(-1)^{(r-1)/2}r$.
Note that
\[
q\equiv
\begin{cases}
1\mod 8,& \text{if }r\equiv 1,7\mod 8,\\
5\mod 8,& \text{if }r\equiv 3,5\mod 8.
\end{cases}
\]

\begin{proof}[Proof of Corollary \ref{mainQ}]
Using the above notation with $K:=\mathbb{Q}$, we will show that
$r\equiv 3,5\mod 8$ implies that
$\pi_r\not\equiv\nu^2\mod\fP^{4e+1}$ with
$\nu\in\cO_{\mathbb{Q}^+}$ and then the result follows from
Theorem \ref{main2}.

Suppose by a contradiction that
$\pi_r\equiv\nu^2\mod\fP^{4e+1}$. By Lemma \ref{sqnorm} with
$K=\mathbb{Q}$, $n=4e+1$ and $e'=e$, we get that
\(
q\equiv v^2\mod 2^5
\) 
for some $v\in\mathbb{Z}$. As $q$ is odd, $v$ is odd. However,
the odd squares modulo $32$ are $\{1,9,17,25\}$, whereas
$q\equiv 5\mod 8$. This gives a contradiction.
\end{proof}

\begin{proof}[Proof of Corollary \ref{example}]
We will check the assumptions in Corollary \ref{mainQ} hold.
Note that in order to check that $2\nmid h_{\mathbb{Q}^+}^+$,
it is sufficient to check that $2\nmid h_{\mathbb{Q}(\z)}$ as
$h_{\mathbb{Q}^+}^+|h_{\mathbb{Q}(\z)}$. This follows by noting
that $\mathbb{Q}(\z)/\mathbb{Q}^+$ is totally ramified at $r$.
Thus, if we take $H$ to be the narrow Hilbert class field of
$\mathbb{Q}^+$, then $H\cap\mathbb{Q}(\z)=\mathbb{Q}^+$ and
$H(\z)$ is a subfield of the Hilbert class field of
$\mathbb{Q}(\z)$, giving the desired divisibility.

By Hasse's Theorem (\cite{Hasse}, Theorem 3.45) one gets that
$h_{\mathbb{Q}(\z)}$ is odd if and only if the relative class
number $h_r^{-}$\footnote{The relative class number $h_r^{-}$
is defined to be
$h_r^{-}:=h_{\mathbb{Q}(\zeta_r)}/h_{\mathbb{Q}^+}$.} is odd.
One can check the tables computed in \cite[p. 412]{Washington}
to see that for the above values of $r$ the associated relative
class number $h_r^{-}$ is odd, so indeed
$2\nmid h_{\mathbb{Q}^+}^+$.
For each of the primes in the statement, $r\equiv 3,5\mod 8$
and the multiplicative order of $2$ modulo $r$ is $r-1$.
Hence $2$ is inert in $\mathbb{Q}^+$, completing the proof.
\end{proof}

\begin{proof}[Proof of Corollary \ref{mainQd}]
Using the above notation with $K:=\mathbb{Q}(\sqrt{d})$, we
will show that the conditions in (i) and (ii) imply that
$\pi_r\not\equiv\nu^2\mod\fP^{4e+1}$ with
$\nu\in\cO_{K^+}$ and then the result follows from
Theorem \ref{main2}. Note that as there is a unique prime
above $2$ in $K$ we have the following two cases.

\textbf{Case 1:} Suppose that $d\equiv 5\mod 8$ and so
$\cO_K=\mathbb{Z}[\frac{1+\sqrt{d}}{2}]$. In this case $2$ is
inert in $K$, giving $e=e'=e(\fP/2)$. Suppose by a contradiction
that $\pi_r\equiv\nu^2\mod\fP^{4e+1}$ with
$\nu\in\cO_{K^+}$. By Lemma \ref{sqnorm} it follows that
$q\equiv v^2\mod 32\cO_K$ for
$v=a+b\frac{1+\sqrt{d}}{2}\in\cO_K$. Hence
\[
q\equiv\left(a+b\frac{1+\sqrt{d}}{2}\right)^2
=\left(a^2+b^2\frac{d-1}{4}\right)
 +\left(b^2+2ab\right)\frac{1+\sqrt{d}}{2}
\mod 32\cO_K.
\]
Comparing coefficients in the integral basis
$1,\frac{1+\sqrt{d}}{2}$, it follows that
\[
\begin{cases}
b^2+2ab\equiv 0\mod 32,\\
a^2+b^2\frac{d-1}{4}\equiv q\mod 32.
\end{cases}
\]
The first equation implies that $b$ is even. If $a$ were even
too, the second equation would imply that $q$ is even, which
is a contradiction. Thus $a^2\equiv 1\mod 8$ and
$b^2\equiv 0,4\mod 8$. Using the second equation we get either
$q\equiv 1\mod 8$ or $q\equiv d\mod 8$, contradicting (i).

\textbf{Case 2:} Suppose that $d\equiv 2,3\mod 4$ and so
$\cO_K=\mathbb{Z}[\sqrt{d}]$. In this case $2$ is totally
ramified in $K$ and we denote by $\fQ$ the unique prime above
$2$ in $K$. Hence $\fQ^2=(2)\cO_K$ and
$e:=e(\fP/2)=2e(\fP/\fQ)=2e'$.

Suppose by a contradiction that
$\pi_r\equiv\nu^2\mod\fP^{4e+1}$ with
$\nu\in\cO_{K^+}$. As $\lceil(4e+1)/e'\rceil=9$,
Lemma \ref{sqnorm} gives
\(
q\equiv v^2\mod\fQ^9
\)
for some $v\in\cO_K$. In particular, as $\fQ^8=(16)\cO_K$,
we have $q\equiv v^2\mod 16\cO_K$. Write
$v=a+b\sqrt{d}$ with $a,b\in\mathbb{Z}$. It follows that
\[
q\equiv\left(a+b\sqrt{d}\right)^2
=a^2+b^2d+2ab\sqrt{d}\mod 16\cO_K.
\]
Comparing coefficients in the integral basis $1,\sqrt{d}$,
it follows that
\[
\begin{cases}
2ab\equiv 0\mod 16,\\
a^2+b^2d\equiv q\mod 16.
\end{cases}
\]
The first equation implies that $8|ab$. As $q$ is odd, $a$ and
$b$ cannot be simultaneously even, so we have two cases.
\begin{itemize}
    \item If $a$ is odd, then $8|b$ and the second equation gives
    $q\equiv a^2\mod 16$. In particular, $q\equiv 1\mod 8$,
    contradicting (i).
    \item If $b$ is odd, then $8|a$ and the second equation gives
    $q\equiv b^2d\mod 16$. In particular, $q\equiv d\mod 8$,
    contradicting (i).
\end{itemize}
Hence (i) and (ii) imply that
$\pi_r\not\equiv\nu^2\mod\fP^{4e+1}$ with
$\nu\in\cO_{K^+}$, so we can conclude the proof by
Theorem \ref{main2}.
\end{proof}

\begin{appendices}
\section{Proof of Theorem \ref{LLES}}\label{appendix}
Theorem \ref{LLES} is a slightly adjusted version of Theorem 9 in \cite{SN} which puts together modularity, level lowering and Eichler-Shimura results to reduce the problem of studying the Frey elliptic curve $E_{(x,y)}$ which is dependent on the putative solution, to studying an elliptic curve $E'$, which is independent of the putative solution. Moreover, Theorem \ref{LLES} gives us a precise description of the behaviour of $E'$ at the good and bad primes, which will be crucial in getting the desired contradiction (Step 3 of the modular method).

\begin{proof}
    
We first observe by Proposition \ref{lemmane} that $E$ has multiplicative reduction outside $S_{K^+}$. After possibly enlarging $B_{K,r}$ we can assume that $E$ is modular by Corollary \ref{modularity3} and that $\rep_{E,p}$ is irreducible by Theorem \ref{irred}. Applying Theorem \ref{ll} and Proposition \ref{lemmane} we see that $\rep_{E,p} \sim \rep_{\mathfrak{f},\varpi}$ for a Hilbert newform $\mathfrak{f}$ of level $\mathcal{N}_p$ and
		some prime $\varpi|p$ of $\mathbb{Q}_{\mathfrak{f}}$.
		Next we reduce to the case when $\mathbb{Q}_{\mathfrak{f}}=\mathbb{Q}$, after possibly enlarging $B_{K,r}$. This step uses standard ideas originally due to Mazur that can be found in \cite[Section 4]{BS}, \cite[Proposition 15.4.2]{HC2}, and so we omit the details.
		\par
		Next, we want to show that there is some elliptic curve $E'/K^+$ of conductor $\mathcal{N}_p$ having the same L-function as $\mathfrak{f}$. We apply
		Lemma \ref{imgi} and by possibly enlarging $B_{K,r}$ get that $E$ has potentially
		multiplicative reduction at ${\mathfrak{P}}$ and $p | \# \rep_{E,p}(I_{\mathfrak{P}})$. The existence of $E'$ follows from Theorem \ref{inertia} after possibly enlarging $B_{K,r}$ to ensure that $p \nmid (\text{Norm}_{K^+/\mathbb{Q}}(\tilde{\mathfrak{\mathfrak{P}}})\pm 1)$.
		By putting all the pieces together we can conclude that there is an elliptic curve $E'/K$ of conductor $\cN_p$ satisfying $\rep_{E,p} \sim \rep_{E',p}$. This proves (i) and (ii).
		
		To prove (iii) we use that $\rep_{E,p} \sim \rep_{E',p}$ for some $E'/K^+$ with conductor  $\mathcal{N}_{p}$. After enlarging $B_{K,r}$ by an effective amount, and possibly replacing $E'$ by an isogenous curve, we may assume that $E'$ has full $2$-torsion over $K^+$. This uses standard ideas which can be found, for example, in \cite[Section IV-6]{Serre}.
		
		Now let $j_{E'}$ be the $j$-invariant of $E'$. As we have already seen, Lemma \ref{imgi} implies $p | \# \rep_{E,p}(I_{\tilde{\mathfrak{P}}})$, hence $p | \# \rep_{E',p}(I_{\tilde{\mathfrak{P}}})$, thus by Lemma \ref{lemma} we get that $E'$ has potentially multiplicative reduction at $\tilde{\mathfrak{P}}$ and so $v_{\tilde{\mathfrak{P}}}(j_{E'})<0$.
	\end{proof}
	
\end{appendices}
\newpage
\printbibliography
\end{document}